\documentclass[10pt]{article}
\usepackage{amssymb, amsfonts, amsthm, mathrsfs}
\usepackage[super]{nth}
\usepackage[pagebackref,colorlinks=true]{hyperref}
\usepackage{graphicx, xcolor}
\usepackage[all]{xy}
\usepackage[leqno]{amsmath}
\usepackage{titlesec}

\theoremstyle{definition}
\newtheorem{thm}{Theorem}[section]

\newtheorem{prop}[thm]{Proposition}

\newtheorem{rmk}[thm]{Remark}

\titleformat{\section}{\Large\bfseries}{\thesection}{.5em}{}

\makeatletter

\def\ccc{\mathbb{C}}
\def\C{\mathbb{C}}

\def\rr{\mathbb{R}}

\def\pp{\mathbb{P}}

\def\pt{\partial}
\def\p{\partial}
\def\bpt{\bar{\pt}}
\def\ddb{\pt\bpt}

\def\bzeta{\bar{\zeta}}

\newcommand{\be}{\begin{equation}}
\newcommand{\bea}{\begin{eqnarray}}
\newcommand{\eea}{\end{eqnarray}}
\newcommand{\ee}{\end{equation}}

\makeatother

\begin{document}

\title{\textbf{The Anomaly flow over Riemann surfaces}}
\author{Teng Fei, Zhijie Huang and Sebastien Picard}

\date{}

\maketitle{}

\abstract{We initiate the study of a new nonlinear parabolic equation on a Riemann surface. The evolution equation arises as a reduction of the Anomaly flow on a fibration. We obtain a criterion for long-time existence for this flow, and give a range of initial data where a singularity forms in finite time, as well as a range of initial data where the solution exists for all time. A geometric interpretation of these results is given in terms of the Anomaly flow on a Calabi-Yau threefold.}

\tableofcontents

\section{Introduction}

Consider a Riemann surface $\Sigma$ and a holomorphic map $\varphi: \Sigma \rightarrow \C \pp^1$ such that $\varphi^* \mathcal{O}(2) = K_\Sigma$. We will call $(\Sigma, \varphi)$ a vanishing spinorial pair. By pulling back sections of $\mathcal{O}(2)$, we will construct three holomorphic $(1,0)$ forms $\mu_1$, $\mu_2$ and $\mu_3$ which can be associated to the pair $(\Sigma,\varphi)$. Given a metric $\omega$ on $\Sigma$, we may take the norm of $\mu_k$, which we will denote by $\| \mu_k \|^2_\omega$, and we will use the notation
\[
\| \mu \|^2_\omega = \| \mu_1 \|^2_\omega + \| \mu_2 \|^2_\omega + \| \mu_3 \|^2_\omega.
\]
For a fixed slope parameter $\alpha'>0$, we introduce the following conformal flow
\be \label{spinorial-pair-flow}
\frac{\p_t \omega(t)}{\omega(t)} = {1 \over \| \mu \|^2_{\omega(t)}} \left( -R_{\omega(t)} + \left| \p \log \| \mu \|_{\omega(t)}^2 \right|_{\omega(t)}^2 \right) + {\alpha' \over 8} \left( \Delta_{\omega(t)} \| \nabla \varphi \|^2_{\omega(t)} - \| \nabla \varphi \|^4_{\omega(t)}  \right).
\ee
The goal of this paper is to take the first steps in building the analytic theory for this equation.
\smallskip
\par We now discuss the motivation for studying the flow (\ref{spinorial-pair-flow}). Given a complex threefold $X$ with non-vanishing holomorphic $(3,0)$ form $\Omega$ and a metric $\omega_0$ satisfying the conformally balanced condition $d (\| \Omega \|_{\omega_0} \omega_0^2) = 0$, the Anomaly flow with trivial gauge bundle is given by
\be \label{af_general}
\p_t ( \| \Omega \|_{\omega} \omega^2 ) = i \ddb \omega - {\alpha' \over 4} {\rm Tr} \, Rm \wedge Rm, \ \ \omega(0) = \omega_0.
\ee
This flow was introduced by Phong, Zhang and the third-named author \cite{phong2017} to study the Hull-Strominger system \cite{hull1986b,strominger1986}. This flow preserves the conformally balanced condition $d (\| \Omega \|_{\omega(t)} \omega(t)^2) = 0$, making it interesting from the point of view of non-K\"ahler geometry, as it provides a deformation path in the space of balanced metrics. Short-time existence of this flow was established in \cite{phong2017} for initial metrics $\omega$ satisfying $|\alpha' Rm(\omega)| \ll 1$. The study of the Anomaly flow has just begun, but progress made in various directions \cite{phong2016b,phong2016c,phong2017b}.
\smallskip
\par Our motivation for studying such non-K\"ahler Calabi-Yau threefolds $X$ comes from theoretical physics. Candelas-Horowitz-Strominger-Witten \cite{candelas1985} proposed to use K\"ahler Calabi-Yau threefolds as torsion-free compactification of superstrings, bridging string theory with K\"ahler Ricci-flat metrics in complex geometry \cite{yau1978}. This idea was later extended by Hull-Strominger \cite{hull1986b,strominger1986} to allow superstrings with torsion, which leads to a system of equations in non-K\"ahler geometry with conformally balanced metrics. Examples of solutions to the Hull-Strominger system include \cite{li2005, fu2007,fu2008,fernandez2009, grantcharov2011, fei2015, fu2009, andreas2012, fernandez2014, fei2015d, halmagyi2016} and references therein. Stationary points of the Anomaly flow are solutions to the Hull-Strominger system.
\smallskip
\par In \cite{fei2017}, we gave a family of solutions to the Hull-Strominger system with infinitely many topological types. These threefolds were first constructed by the first-named author \cite{fei2015b,fei2016}, and they generalize a construction of Calabi \cite{calabi1958} and Gray \cite{gray1969}. Given a vanishing spinorial pair $(\Sigma,\varphi)$ together with a hyperk\"ahler manifold $M$, the generalized Calabi-Gray construction gives a non-K\"ahler threefold $X$ which is a fibration $p: X \rightarrow \Sigma$ with fiber $M$. In this paper, we study the Anomaly flow on generalized Calabi-Gray manifolds, and show that the flow (\ref{spinorial-pair-flow}) appears as a reduction of the Anomaly flow (\ref{af_general}) with the ansatz $\omega_f = e^{2f} \hat{\omega} + e^f \omega'$ for an arbitrary function $f \in C^\infty(\Sigma,\mathbb{R})$, where $\hat{\omega}$ is a canonical metric on $\Sigma$ constructed from $(\Sigma,\varphi)$ and $\omega'$ is positive in the fiber directions.
\smallskip
\par There is another setting where the Anomaly flow on a fibration reduces to a scalar equation on the base. In \cite{fu2007,fu2008}, Fu and Yau found the first compact non-K\"ahler solutions to the Hull-Strominger system by studying threefolds constructed by Calabi-Eckmann-Goldstein-Prokushkin \cite{calabi1953b,goldstein2004}. These threefolds are torus fibrations over a $K3$ surface. Fu and Yau wrote down an ansatz which allowed them to reduce the Hull-Strominger system to a fully nonlinear scalar PDE on the $K3$ surface. The equation discovered by Fu-Yau has a rich structure, and was further studied in \cite{phong2016d,phong2016a,phong2017c}. If we start the Anomaly flow with the Fu-Yau ansatz, the flow also descends to a parabolic equation on the base. In \cite{phong2016b}, this flow is studied when the initial data is taken such that $|\alpha' Rm|$ is initially small. We call this the large radius limit, as the ansatz involves a large multiple of the metric on the base. The condition $|\alpha' Rm| \ll 1$ is shown to be preserved along the flow, and the long-time existence of the flow is established. Our main result is the analogous statement in our setting.

\begin{thm} \label{af_collapse_thm}
Suppose $|\alpha' Rm(\omega_f)| \ll 1$, and start the Anomaly flow on a generalized Calabi-Gray manifold $p:X \rightarrow \Sigma$ with initial metric $\omega_f = e^{2f} \hat{\omega} + e^f \omega'$. Then the flow exists for all time and as $t \rightarrow \infty$,
\[
{\omega_f \over {1 \over 3!} \int_X \| \Omega \|_{\omega_f} \, \omega_f^3} \rightarrow p^* \omega_\Sigma,
\]
smoothly, where $\omega_\Sigma = q_1^2 \, \hat{\omega}$ is a smooth metric on $\Sigma$ associated to the vanishing spinorial pair $(\Sigma,\varphi)$. Here $q_1>0$ is the first eigenfunction of the operator $-\Delta_{\hat{\omega}} - \| \nabla \varphi \|^2_{\hat{\omega}}$. Furthermore, $(X,{\omega_f \over \frac{1}{3!} \int_X \| \Omega \|_{\omega_f} \, \omega_f^3})$ converges to $(\Sigma,\omega_\Sigma)$ in the Gromov-Hausdorff topology.
\end{thm}
\par Thus initial metrics in our ansatz satisfying $|\alpha' Rm(\omega_f)| \ll 1$ at the initial time will collapse the hyperk\"ahler fibers under a normalization of the Anomaly flow. The metrics on the base will converge smoothly to a metric $\omega_{\Sigma}$ which can be associated to the vanishing spinorial pair $(\Sigma, \varphi)$.
\smallskip
\par Theorem \ref{af_collapse_thm} can be compared to the phenomenon of collapsing in the K\"ahler-Ricci flow, as pioneered by Song-Tian \cite{song2007,song2012} and further explored by several others \cite{tosatti2014,fong2015,song2016,tosatti2015e,gill2014,zhang2017}. In this case, there is a general theory of collapsing of Calabi-Yau fibrations over K\"ahler manifolds $B$. The limiting metric $\omega_B$ which appears on the base $B$ in the limit of the flow is a twisted K\"ahler-Einstein metric which solves ${\rm Ric}(\omega_B) = - \omega_B + \omega_{WP}$, where $\omega_{WP}$ is the Weil-Petersson metric of the fibration. Other flows in complex geometry, such as the Chern-Ricci flow and the conical K\"ahler-Ricci flow, also exhibit collapsing behavior \cite{tosatti2015d,fang2016,zheng2017,edwards2017}.
\smallskip
\par In our case, the limiting metric $\omega_\Sigma$ on the base $\Sigma$ has curvature
\[
- i \ddb \log \omega_\Sigma = - (2 \eta q_1^{-2}) \omega_\Sigma + \omega_{WP} + 2 i q_1^{-2} \p q_1 \wedge \bar{\p} q_1.
\]
Here $0<\eta = -\lambda_1/2$, where $\lambda_1$ is the first eigenvalue of the operator $-\Delta_{\hat{\omega}} - \| \nabla \varphi \|^2_{\hat{\omega}} $. Thus the curvature of $\omega_\Sigma$ splits into a negative part proportional to $\omega_\Sigma$ and a positive part which has contributions from the Weil-Petersson form $\omega_{WP} \geq 0$. The size of the first eigenvalue of $-\Delta_{\hat{\omega}}  - \| \nabla \varphi \|^2_{\hat{\omega}}$ was estimated by the first and second-named authors in \cite{fei2017b}.
\smallskip
\par Given this motivation, we now return to discussing the evolution equation (\ref{spinorial-pair-flow}) of a metric $\omega(t)$ on a Riemann surface. We will construct a reference metric $\hat{\omega} = i \hat{g}_{\bar{z} z} dz \wedge d \bar{z}$ on the pair $(\Sigma,\varphi)$ such that the conformal factor of $\omega(t) = e^{f(t)} \hat{\omega}$ evolves by
\be \label{flow-ef}
\p_t e^f =  \hat{g}^{z \bar{z}} \p_z \p_{\bar{z}} (e^f + {\alpha' \over 2}  \kappa e^{-f} ) - \kappa  (e^f + {\alpha' \over 2} \kappa e^{-f} ),
\ee
where $\kappa \in C^\infty(\Sigma,\mathbb{R})$ is a given function such that $\kappa \leq 0$. In fact, $\kappa$ is the Gauss curvature of the metric $\hat{\omega}$. We begin by noting that the equation is parabolic, and thus a solution exists for a short-time.

\begin{thm}
Given an initial smooth function $h: \Sigma \rightarrow \mathbb{R}$, there exists $T>0$ such that a smooth solution of the evolution equation (\ref{flow-ef}) exists on $[0,T)$ and $f(x,0) = h(x)$.
\end{thm}

The question of long-time existence of solutions is the main focus of this work. We will see that the behavior of this evolution equation is sensitive to the balance of the size of $e^f$ against its reciprocal $e^{-f}$. Indeed, for initial data where $e^f$ is large initially, the term $e^f$ dominates and the flow exists for all time. However, for initial data where $e^f$ is small initially, the nonlinear reciprocal term $e^{-f}$ dominates and the flow develops a singularity in finite time.
\smallskip
\par Theorem \ref{af_collapse_thm} corresponds to studying the flow (\ref{flow-ef}) with large initial data $e^f(x,0)$. This result will be discussed as a pure PDE problem in \S \ref{collapsing-section}, and stated without geometric interpretation as Theorem \ref{u>0-conv}. The condition on the initial data can be viewed as a perturbative hypothesis, and the equation is shown to evolve almost linearly in this case. In terms of the formalism of equation (\ref{spinorial-pair-flow}), the statement of Theorem \ref{af_collapse_thm} becomes the following:

\begin{thm} \label{spinorial-pair-flow-thm}
Start the flow (\ref{spinorial-pair-flow}) with an initial metric satisfying $\omega(0) \geq \dfrac{\sqrt{\alpha'}}{2} \| \nabla \varphi \|_{\hat{\omega}} \, \hat{\omega}$. Then the flow exists for all time, and
\[ {\omega(t) \over \int_\Sigma \omega(t)} \rightarrow \frac{q_1\hat{\omega}}{\int_\Sigma q_1\hat{\omega}}\]
smoothly as $t \rightarrow \infty$. As before, $q_1>0$ is the first eigenfunction of the operator $-\Delta_{\hat{\omega}} - \| \nabla \varphi \|^2_{\hat{\omega}}$.
\end{thm}

Without the large initial data hypothesis, we are in the non-perturbative regime, and the flow can behave very nonlinearly. We will show in Proposition \ref{finitesing} that finite time singularities occur for a certain range of initial data. However, the only possibility is that $e^f$ reaches zero in finite time, and we have the following long-time existence criterion.

\begin{thm} \label{long-time-criterion}
Suppose a solution to the evolution equation (\ref{flow-ef}) exists on a time interval $[0,T)$ with $T<\infty$. Then $e^f$ remains bounded on $[0,T)$. Furthermore, if $\sup_{\Sigma \times [0,T)} e^{-f} < \infty$, then the solution can be extended to an interval $[0,T+\epsilon)$ for some $\epsilon>0$.
\end{thm}

For future developments, we will introduce in \S \ref{general-properties} an energy functional which is monotone decreasing along the flow, state a conservation law, and give a local estimate on $e^f$ which holds for arbitrary initial data. As this nonlinear evolution equation arises naturally from geometry and mathematical physics, we hope to continue its study and further develop the relevant analytic tools and techniques.
\smallskip
\par The outline of the paper is as follows. In \S \ref{gcg-af}, we recall the generalized Calabi-Gray construction and establish the geometric setup for our fibration $p: X \rightarrow \Sigma$, and derive the evolution equation (\ref{flow-ef}). In \S \ref{general-properties}, we prove the long-time existence criterion (Theorem \ref{long-time-criterion}) and discuss general features of the evolution equation. In \S \ref{collapsing-section}, we obtain growth estimates and study the limiting behavior of solutions in the case of large initial data. These analytic results are then interpreted geometrically to obtain Theorem \ref{af_collapse_thm}. Finally, in \S \ref{future-dir}, we mention some questions and further directions related to this work.
\medskip
\par \noindent\textbf{Acknowledgements:} We would like to thank Prof. D.H. Phong and Prof. S.-T. Yau for their guidance and support. We also thank Pei-Ken Hung and Xiangwen Zhang for helpful discussions. The third-named author was supported in part by National Science Foundation grant DMS-12-66033.

\section{Anomaly flow on generalized Calabi-Gray manifolds} \label{gcg-af}

\subsection{Setup and notation} \label{gcg-setup}

In this section, we review the construction of generalized Calabi-Gray manifolds \cite{fei2016, fei2016b}. We give the outline of the construction to establish notation.
\smallskip
\par Let $\Sigma$ be a Riemann surface and $\varphi: \Sigma \rightarrow \C \pp^1$ be a holomorphic map such that $\varphi^* \mathcal{O}(2) = K_\Sigma$. We will use coordinates $\zeta=z_2/z_1$ and $\eta = z_1/z_2$ on $\C\pp^1$. We write $\varphi=(\alpha,\beta,\gamma)$, which using the stereographic projection can written as
\begin{equation}\label{stereographic}
(\alpha,\beta,\gamma)= \left(\frac{1-|\zeta|^2}{1+|\zeta|^2},\frac{\zeta+\bzeta}{1+|\zeta|^2},\frac{i(\bzeta-\zeta)}{1+|\zeta|^2}\right).
\end{equation}
We will use the sections $\phi_1$, $\phi_2$ and $\phi_3$ of $\mathcal{O}(2) = T \C \pp^1$ defined by
\[\begin{split}\phi_1&=\zeta\frac{\p}{\p\zeta}=-\eta\frac{\p}{\p\eta},\\ \phi_2&=\frac{1}{2}(\zeta^2-1)\frac{\p}{\p\zeta}=\frac{1}{2}(\eta^2-1)\frac{\p}{\p\eta},\\ \phi_3&=-\frac{i}{2}(\zeta^2+1)\frac{\p}{\p\zeta}=\frac{i}{2}(\eta^2+1)\frac{\p}{\p\eta}.\end{split}\]
Let $\mu_1 = \varphi^* \phi_1$, $\mu_2 = \varphi^* \phi_2$ and $\mu_3 = \varphi^* \phi_3$ be holomorphic 1-forms on $\Sigma$. The expression
\[\widehat{\omega}=i(\mu_1\wedge \overline{\mu_1} +\mu_2\wedge \overline{\mu_2} +\mu_3\wedge \overline{\mu_3})\]
defines a K\"ahler metric on $\Sigma$. This metric is canonical in the sense that it is a generalization of the Weierstrass representation of minimal surfaces in $\rr^3$. We rescale $\hat{\omega}$ such that $\int_\Sigma \hat{\omega} = 1$ and we will fix $\hat{\omega}$ as our background metric from now on. We will use the notation
\[
\hat{\omega} = i \hat{g}_{\bar{z} z} \, dz \wedge d \bar{z}, \ \ \hat{g}^{z \bar{z}} = \hat{g}_{\bar{z} z}^{-1}, \ \ \kappa  = - \hat{g}^{z \bar{z}} \p_z \p_{\bar{z}} \log \hat{g}_{\bar{z} z}.
\]
We denote $|\p f|^2 = \hat{g}^{i \bar{j}} \p_i f \p_{\bar{j}} f$ for a function $f \in C^\infty(\Sigma, \mathbb{R})$, and $\Delta_{\hat{\omega}}$ will denote the real Laplacian of the metric $\hat{\omega}$, so that $2 i \ddb u = \Delta_{\hat{\omega}} u \, \hat{\omega}$. The Fubini-Study metric on $\C\pp^1$ will be denoted
\[
\omega_{FS} = {2 \over (1+|\zeta|^2)^2} i d \zeta \wedge d \bar{\zeta}.
\]
A standard computation (see e.g. \cite{fei2017}) gives the following identities
\be \label{kappa-identity}
- \varphi^* \omega_{FS} = \kappa \, \hat{\omega} = - { \| \nabla \varphi \|_{\hat{\omega}}^2 \over 2} \hat{\omega},
\ee
where $\| \nabla \varphi \|_\omega^2 = 2 (g_{FS}{}_{\bar{\zeta} \zeta} \circ \varphi) g^{z \bar{z}} \varphi_z \, \overline{\varphi_z}$ for a metric $\omega = i g_{\bar{z} z} dz \wedge d \bar{z}$, and
\be \label{abc-indakernel}
(\hat{g}^{z \bar{z}} \p_z \p_{\bar{z}} - \kappa) \alpha = (\hat{g}^{z \bar{z}} \p_z \p_{\bar{z}} - \kappa) \beta = (\hat{g}^{z \bar{z}} \p_z \p_{\bar{z}} - \kappa) \gamma = 0.
\ee
Next, we introduce the 4-torus with flat metric $(T^4,g)$, which we view as a compact hyperk\"ahler 4-manifold with complex structures $I$, $J$ and $K$ satisfying $I^2=J^2=K^2=-1$ and $IJ=K$. Write $\omega_I=g(I \cdot, \cdot)$, $\omega_J=g(J \cdot, \cdot)$ and $\omega_K=g(K \cdot, \cdot)$ for the corresponding K\"ahler forms. For any point $(\alpha,\beta,\gamma) \in S^2$, we have a compatible complex structure $\alpha I+\beta J+\gamma K$ on $T^4$, whose associated K\"ahler form is given by $\alpha\omega_I+\beta\omega_J+\gamma\omega_K$. Let $X= \Sigma \times T^4$ be the product manifold with a twisted complex structure $J_0 = j_\Sigma \oplus(\alpha I+\beta J+\gamma K)$, where we now let $\varphi = (\alpha,\beta,\gamma)$ as before. We can view the threefold $X$ as the pullback of the twistor space of $T^4$ \cite{hitchin1987}.
\medskip
\par A result in \cite{fei2015b,fei2016} shows that $X$ is non-K\"ahler with trivial canonical bundle, with trivializing holomorphic $(3,0)$ form $\Omega$ given by
\[ \Omega = \mu_1 \wedge \omega_I + \mu_2 \wedge \omega_J + \mu_3 \wedge \omega_K. \]
For future use, we compute the Weil-Petersson metric $\omega_{WP}$ of the fibration $X \rightarrow \Sigma$. We use the notation $N_y$ for the hyperk\"ahler fiber over $y \in \Sigma$. For a local family $\Psi_y \in H^0(N_y,K_{N_y})$ of non-vanishing sections varying holomorphically in $y$, we define
\[ \omega_{WP} = - i \ddb \log \left( \int_{N_y} \Psi_y \wedge \bar{\Psi}_y \right). \]
This definition gives a well-defined metric (see e.g. \cite{tosatti2015c}). In our case, we may take $\Psi_y$ to be $\Omega(\p_y, \cdot, \cdot)$ for a local coordinate $y$ on $\Sigma$ such that $\varphi^* \p_\zeta = dy$. In general, for an arbitrary coordinate $z$ we have $\varphi^* \p_\zeta = f(z) dz$ where $f$ is holomorphic and nowhere vanishing, and we let $y=F(z)$ where $F'(z)=f(z)$. Then
\[ \Psi_y = \varphi(y) \, \omega_I + {1 \over 2}(\varphi^2(y)-1) \omega_J - {i \over 2} (\varphi^2(y) +1) \omega_K, \]
and $\Psi_y \wedge \bar{\Psi}_y = (1+|\varphi|^2)^2 \, {\rm vol}_{T^4}$. It follows that $\omega_{WP} = \varphi^* \omega_{FS}$.

\subsection{The Anomaly flow}

Our ansatz on the threefold $X$ will be given by
\[
\omega_f = e^{2f} \hat{\omega} + e^f \omega', \ \ \omega' = \alpha\omega_I+\beta\omega_J+\gamma\omega_K,
\]
where $f \in C^\infty(\Sigma,\mathbb{R})$. It was computed in previous work \cite{fei2015d,fei2016} that
\be \label{balanced-class}
\| \Omega \|_{\omega_f} = e^{-2f}, \ \ \| \Omega \|_{\omega_f} \, \omega_f^2 = 2 {\rm vol}_{T^4} + 2 e^f \hat{\omega} \wedge \omega',
\ee
and
\be \label{volume-omega_f}
\int_X \| \Omega \|_{\omega_f} \omega_f^3 = \int_X e^{2f} \hat{\omega} \wedge (6 {\rm vol}_{T^4}).
\ee
Furthermore, $d (\| \Omega \|_{\omega_f} \, \omega_f^2) = 0$ for any arbitrary function $f$. Next, we introduce $u \in C^\infty(\Sigma,\mathbb{R})$ to be defined by
\be \label{u-defn}
u = e^f + {\alpha' \over 2} e^{-f} \kappa.
\ee
Solving for $e^f$ gives
\[e^f=\frac{1}{2}(u+\sqrt{u^2-2\alpha'\kappa})>0. \]
In \cite{fei2017} (see also \cite{fei2015d,fei2016}), the following identity was derived
\[
i \ddb \omega - {\alpha' \over 4} \, {\rm Tr} \, Rm(\omega_f) \wedge Rm(\omega_f) = (i \ddb u - \kappa u \hat{\omega}) \wedge \omega'.
\]
Therefore, the Anomaly flow (\ref{af_general}) reduces to
\[
(\p_t e^f  \hat{\omega}) \wedge \omega' = {1 \over 2} (i \ddb u - \kappa u \hat{\omega}) \wedge \omega'.
\]
From here, we obtain equation (\ref{flow-ef}) for a scalar $e^f$ on the base $\Sigma$, after rescaling $e^f(x,s) = e^f(x,2t)$ to remove the factor of a half. To obtain the evolution of $\omega(t) = e^{f(t)} \hat{\omega}$ given in (\ref{spinorial-pair-flow}), it suffices to use the identities $\| \mu \|^2_{\omega(t)} = e^{-f}$, $\| \nabla \varphi \|^2_{\omega(t)} = - 2 \kappa e^{-f}$, and
\[
R_{\omega(t)} =  -e^{-f} \hat{g}^{z \bar{z}} f_{\bar{z} z} + e^{-f} \kappa, \ \ \ \ |\p \log \| \mu \|_{\omega(t)}^2 |_{\omega(t)}^2 = e^{-f} \hat{g}^{z \bar{z}} f_z f_{\bar{z}} .
\]
The function $u(t)$ will play an important role in the analysis of this reduction of the Anomaly flow. Its evolution is given by
\be \label{flow-u}
\p_t u = (1 - {\alpha' \over 2} \kappa e^{-2f})  \hat{g}^{z \bar{z}} u_{\bar{z} z}  - \kappa (1 - {\alpha' \over 2} \kappa e^{-2f}) u .
\ee
We will also use the notation
\be \label{flow-u-2}
\p_t u = a^{z \bar{z}} u_{\bar{z} z} - (\kappa a^{z \bar{z}} \hat{g}_{\bar{z} z}) \, u, \ \ a^{z \bar{z}} = (1 - {\alpha' \over 2} \kappa e^{-2f}) \hat{g}^{z \bar{z}}.
\ee

\section{General properties of the evolution equation}
\label{general-properties}
Recall that we work on a Riemann surface $(\Sigma, \hat{\omega})$ equipped with a map $\varphi: \Sigma \rightarrow \ccc\pp^1$, where the curvature of the reference metric $\hat{\omega}$ is denoted $\kappa$ and satisfies $\kappa = - \frac{1}{2} \| \nabla \varphi \|^2_{\hat{\omega}}$. In particular $\kappa \leq 0$, and $\kappa =0$ at the branch points of $\varphi$. In this section, we will discuss some basic properties of the evolution equation (\ref{flow-ef}) with fixed $\alpha'>0$. This equation is parabolic since it can be written as
\bea \label{f-dot}
\left(\p_t - (1 - \frac{\alpha'}{2} \kappa e^{-2f}) \hat{g}^{z \bar{z}} \p_z \p_{\bar{z}} \right) f &=& (1 + {\alpha' \over 2} \kappa e^{-2f} ) |\p f|^2 - \alpha' e^{-2f} {\rm Re} \{ \hat{g}^{z \bar{z}} \p_z  \kappa \p_{\bar{z}} f \}  \nonumber\\
&& - \kappa - {\alpha' \over 2} e^{-2f}  ( \kappa^2 - \hat{g}^{z \bar{z}} \kappa_{\bar{z} z} ),
\eea
hence we may assume a solution exists on $[0,T)$ for some $T>0$, for any smooth initial data $f(x,0)$.

\subsection{A criterion for extending the flow}
First, we are going to prove the second part of Theorem \ref{long-time-criterion}, which states that as long as $e^{-f}$ stays bounded, the flow can be continued.
\smallskip
\par
Indeed, if we look at the evolution equation of $u = e^f + \dfrac{\alpha'}{2} e^{-f}\kappa$ given in (\ref{flow-u-2}), we see that if $\sup_{\Sigma \times [0,T)} e^{-f}  < \infty$, there exists $\Lambda > 0$ such that
\[
\hat{g}^{z \bar{z}} \leq a^{z \bar{z}} \leq \Lambda \hat{g}^{z \bar{z}} \ \  \mbox{and} \ \ 0 \leq (-\kappa a^{z \bar{z}} \hat{g}_{\bar{z} z}) \leq \Lambda.
\]
Thus the evolution equation of $u$ is uniformly parabolic on $[0,T)$. By applying the maximum principle to $e^{\Lambda t} u$, we see that $u(x,t) \leq e^{\Lambda t} u(x,0)$. Since $e^{-f}$ and $u$ are both bounded above, we conclude that $e^f$ is bounded, and so $\| u \|_{L^\infty(\Sigma \times [0,T))} < \infty$ and $\| f \|_{L^\infty(\Sigma \times [0,T))} < \infty$. By the Krylov-Safonov estimate \cite{krylov1980,krylov1981}, for some $0<\alpha<1$ we have
	$$||u||_{C^{\alpha, \alpha/2}} < \infty$$
	which in turn implies that $e^f , a^{z \bar z} \in C^{\alpha, \alpha/2}$. By the parabolic Schauder estimates \cite{krylov1996}, we get that $$||u||_{C^{2+\alpha, 1+\alpha/2}} < \infty.$$
A bootstrap argument gives higher order estimates. Hence we can find a subsequence $t_i \to T$ such that
	$ u(\cdot, t_i) \to u_T $ and $e^{f(\cdot, t_i)} \to e^{f_T}$ for some function $e^{f_T}$ and $u_T$. Since the convergence is at least in $L^\infty$, we know that $u_T = e^{f_T} + \frac{\alpha'}{2} e^{-f_T} \kappa$ still holds.
	Now we can continue the flow of $u$ using the initial data $u_T$. To solve for $e^f$ from $u$, the only condition we need is $u>0$ wherever $\kappa = 0$. This condition is satisfied at $t=T$ since $e^{f_T} \geq \delta > 0$, and it is an open condition, so it must be also satisfied for some small $\varepsilon > 0$. This proves that the flow can be extended past $T$.
\smallskip
\par To complete the proof of Theorem \ref{long-time-criterion}, we will show that $e^f$ cannot go to infinity in finite time. Let $\Lambda> |\kappa| + \dfrac{\alpha'}{2} \kappa^2$, and consider $e^{-\Lambda t} u$. Then
\[ (\p_t - a^{z \bar{z}} \p_z \p_{\bar{z}}) e^{-\Lambda t} u = (-\Lambda - \kappa + \frac{\alpha'}{2} \kappa^2 e^{-2f} )e^{-\Lambda t} u. \]
Suppose $e^{-\Lambda t} u$ attains its maximum on $\Sigma \times [0,T]$ at $(p,t_0)$ with $t_0>0$ and $u(p,t_0)>0$. If $e^{f(p,t_0)} \leq 1$, then since $u \leq e^f$ we have
\[u(x,t) \leq u(p,t_0)e^{\Lambda(t-t_0)} \leq e^{\Lambda t}.\]
for all $(x,t) \in \Sigma \times [0,T]$. On the other hand, suppose $e^{f(p,t_0)} \geq 1$. Then at $(p,t_0)$, we have the inequality
\[ (\p_t - a^{z \bar{z}} \p_z \p_{\bar{z}}) e^{-\Lambda t} u \leq (-\Lambda - \kappa + \frac{\alpha'}{2} \kappa^2)e^{-\Lambda t} u < 0, \]
which is a contradiction to the maximum principle. It follows that
\[ e^{-\Lambda t} u \leq 1+\| u(x,0) \|_{L^\infty(\Sigma)}. \]
Hence $u$ is bounded above on finite time intervals, and since $u \leq C$ implies $e^f \leq C + \dfrac{\alpha'}{2}|\kappa| e^{-f}$, we see that $e^f$ must also be bounded on finite time intervals.

\subsection{Monotonicity of energy}

We first note that we will often omit the background volume form $\hat{\omega}$ when integrating scalars. Define the energy of $u$ to be
$$I(u) = \frac{1}{2} \int_{\Sigma} |\p u|^2 + \frac{1}{2} \int_{\Sigma}\kappa u^2. $$
\begin{prop}
	Along the flow (\ref{flow-u}), the energy $I(u)$ is monotone non-increasing.
\end{prop}
\begin{proof}
	Differentiating $I(u)$ with respect to $t$, we have
	\bea \label{ddt-energy}
		{d \over dt} I(u) &=&  \int_{\Sigma} - \dot{u} \hat g^{z \bar z} u_{\bar z z } +\int_{\Sigma} \kappa  \dot{u} u  \nonumber \\
		&=&  - \int_{\Sigma} (1 - {\alpha' \over 2} \kappa e^{-2f} ) (\hat g^{z\bar z}u_{\bar z z} -\kappa u)^2  \leq 0 \nonumber
	\eea
	Hence, along the flow, the energy $I(u)$ is monotone non-increasing.
\end{proof}

\subsection{Conservation laws}\label{conservation-law}
\begin{prop}
	Let $\varphi$ be any function in the kernel of $L(w) = -\hat g^{z\bar z} w_{\bar z z}+\kappa w$. Then the integral
	$\int_{\Sigma} e^f \varphi $ is constant along the flow. In particular, $\int_{\Sigma}e^f\alpha, \int_{\Sigma} e^f\beta, \int_{\Sigma}e^f\gamma $ are preserved along the flow.
\end{prop}
\begin{proof}
	Taking the derivative of $\int_{\Sigma} e^f \varphi $ with respect to $t$, we have
	\bea
		{d\over dt} \int_{\Sigma} e^f \varphi &=& \int_{\Sigma} (\p_t e^f ) \varphi \nonumber \\
		&=& \int_{\Sigma} (\hat{g}^{z\bar z} u_{\bar z z} - \kappa u) \varphi  \nonumber \\
		&=& \int_{\Sigma} (\hat{g}^{z\bar z} \varphi_{\bar z z} - \kappa \varphi) u. \nonumber
	\eea
	This vanishes because $\varphi$ is in the kernel of the operator $L(w) = -\hat{g}^{z\bar z} w_{\bar z z} + \kappa w$. As previously noted (\ref{abc-indakernel}), $\alpha, \beta, \gamma$ are all in the kernel, so the last statement follows readily.
\end{proof}

For later use, let us denote the vector
\[(\int_\Sigma e^f\alpha, \int_\Sigma e^f\beta, \int_\Sigma e^f\gamma)\in\rr^3\]
by $V$, which is a constant vector along the flow. Therefore we have
\be \label{L1 est}
    \int_\Sigma e^f \geq \int_\Sigma e^f (\alpha,\beta,\gamma)\cdot\frac{V}{|V|}=|V|,
\ee
as long as the flow exists. As a consequence, if we start the flow with initial data such that $|V|>0$, then automatically we have a lower bound of $\int e^f$.

The conservation laws presented here arise from the fact that the Anomaly flow preserves the conformally balanced cohomology class $[ \| \Omega \|_\omega \, \omega^2 ] \in H^4(X;\mathbb{R})$. In the case of generalized Calabi-Gray manifolds with our ansatz, the de Rham conformally balanced cohomology class is parameterized exactly by the vector $V$, as can be seen by the expression (\ref{balanced-class}).

We conjecture that there are no solutions to the Hull-Strominger system, or equivalently stationary points of the Anomaly flow, with $V=0$. We can show this conjecture is true for certain special hyperelliptic genus 3 curves. The general case can be rephrased as a function theory problem on Riemann surfaces.

\subsection{Finite time blow up}

\begin{prop}\label{finitesing}
	If the $L^1$ norm of $e^{f(\cdot, 0)}$ is initially sufficiently small such that $$\sup_{\Sigma}(-\kappa) \cdot \left( \int_{\Sigma} e^{f(\cdot, 0)} \right)^2 < 8 \alpha' \pi^2(g-1)^2,$$ where $g$ is the genus of $\Sigma$, then the flow will develop a finite time singularity.
\end{prop}
\begin{proof}
By the Gauss-Bonnet Theorem, we have
	\[-\int_\Sigma\kappa=4\pi(g-1).\]
Integrating equation (\ref{flow-ef}) gives
\be \label{evol-int-e^f}
{d \over dt} \left(\int_\Sigma e^f\right)=\int_\Sigma (-\kappa)e^f-\frac{\alpha'}{2}\int_\Sigma \frac{\kappa^2}{e^f}.
\ee
By the Cauchy-Schwarz inequality,
\[\int_\Sigma \frac{\kappa^2}{e^f}\cdot\int_\Sigma e^f\geq\left(\int_\Sigma-\kappa\right)^2=16\pi^2(g-1)^2.\]
If we denote by $A(t)$ the total integral of $e^f$ at time $t$, i.e.,
\[A(t)=\int_\Sigma e^{f(\cdot, t)},\]
and let $K=\sup_X(-\kappa)>0$, then
\[{d\over dt} A \leq   K A-\frac{8 \alpha'\pi^2(g-1)^2}{A}.\]
Equivalently, we have
\[
	{d\over dt} A^2 \leq  2 K A^2 - 16 \alpha' \pi^2(g-1)^2.
\]
From this inequality we see that if $A(0)$ is sufficiently small such that
\[
K A(0)^2<  8\alpha'\pi^2(g-1)^2,
\]
then $A(t)$ is decreasing in $t$, hence we have the bound
\[\int_\Sigma e^{f(\cdot,t)}\leq\int_\Sigma e^{f(\cdot,0)}.\]
In fact we have the estimate
\be \label{est-of-A}
K A(t)^2 \leq 8\alpha'\pi^2(g-1)^2-e^{2 K t}\left(8\alpha'\pi^2(g-1)^2-KA(0)^2\right)
\ee
and the Anomaly flow develops singularity at a finite time.
\end{proof}

The above calculation allows us to give an estimate of the maximal existence time $T$. Let \[E=8\alpha'\pi^2(g-1)^2-KA(0)^2>0.\]
Combining (\ref{est-of-A}) with (\ref{L1 est}), we get
\[K|V|^2<8\alpha'\pi^2(g-1)^2-e^{2KT}E,\]
hence we get the estimate
\[
    T < \frac{1}{2K}\left(\log(8\alpha'\pi^2(g-1)^2-K|V|^2)-\log E\right).
\]

In terms of the intrinsic formulation of our flow (\ref{spinorial-pair-flow}), we just proved the following geometric picture. If we start the flow on a Riemann surface with small initial area, then the area will keep decreasing and before the area reaches zero, the flow will develop finite-time singularities at certain points on the surface. This behavior is in contrast with the Ricci flow.

\begin{rmk}
Though we may think of the Anomaly flow as a generalization of the K\"ahler-Ricci flow on non-K\"ahler Calabi-Yau's with correction terms, there is a fundamental difference. In the K\"ahler-Ricci flow, it is well-known that \cite{cao1985, tsuji1988, tsuji1994, tian2006} that the maximal existence time depends only on the initial K\"ahler class. However for Anomaly flow, our examples shows that even if we fix the de Rham conformally balanced class, the behavior of the flow is sensitive to the initial representative of the cohomology class.

To be precise, choose $e^{f_1}$ small as the first initial data such that the Anomaly flow blows up at finite time as a consequence of Proposition \ref{finitesing}. Let $e^{f_2}=e^{f_1}+Mq_1$ be the second set of initial data, where $q_1$ is the first eigenfunction of the operator $-\Delta_{\hat{\omega}}+2\kappa$ and $M$ is a large positive constant such that $e^{2f_2}>-\alpha'\kappa/2$, or equivalently, the corresponding $u$ is positive. By Theorem \ref{u>0-conv}, to be shown below, we have long-time existence of the flow using this initial condition. Notice that the de Rham conformal balanced cohomology class $[ \| \Omega \|_{\omega_f} \, \omega_f^2 ] \in H^4(X;\mathbb{R})$ is parameterized by three integrals
\[\left(\int_\Sigma e^f\alpha,\int_\Sigma e^f\beta,\int_\Sigma e^f\gamma\right).\]
Our construction implies
\[\left(\int_\Sigma e^{f_1}\alpha,\int_\Sigma e^{f_1}\beta,\int_\Sigma e^{f_1}\gamma\right)=\left(\int_\Sigma e^{f_2}\alpha,\int_\Sigma e^{f_2}\beta,\int_\Sigma e^{f_2}\gamma\right),\]
as $q_1$ and $\{\alpha,\beta,\gamma\}$ are eigenfunctions of the same self-adjoint operator $-\Delta_{\hat{\omega}}+2\kappa$ with distinct eigenvalues. Therefore we can construct two sets of initial data with same de Rham cohomology class such that the first develops a finite time singularity and the second has long-time existence.
\end{rmk}


\subsection{A local estimate in time}
In this section, we prove an estimate on $e^f$ which is independent of the initial data.

\begin{prop}
Suppose $e^f$ solves (\ref{flow-ef}) on $[0,T)$. For any $0<T_0<T$,
\[
\sup_{\Sigma \times [T_0,T]} e^{f} \leq C(T_0,T) \max \bigg\{1 , \int_0^T \int_{\Sigma} e^f \bigg\},
\]
where $C(T_0,T)$ only depends on $T_0$, $T$, $(\Sigma,\hat{\omega})$, $\kappa$, and $\alpha'$.
\end{prop}

\begin{rmk}
Being independent of the initial data, this estimate can be translated in time. It follows that $\int_\Sigma e^f$ locally controls the pointwise behavior of $e^f$ around any point in time.
\end{rmk}

\begin{proof}
First, we compute using (\ref{f-dot})
\begin{equation}
\begin{aligned}
{1 \over p} {d \over dt} \int_\Sigma e^{pf} &= \int_\Sigma e^{pf} (1 - {\alpha' \over 2} e^{-2f} \kappa) \, i \ddb f + \int_\Sigma e^{pf} (1 + {\alpha' \over 2} \kappa e^{-2f}) |\p f|^2 \nonumber\\
& - \alpha' \int_\Sigma e^{pf} e^{-2f} {\rm Re} \{ \hat{g}^{z \bar{z}} \p_z \kappa \p_{\bar{z}} f \}  -  \int_\Sigma e^{pf} \kappa  - {\alpha' \over 2}  \int_\Sigma e^{pf} e^{-2f} (\kappa^2 - \hat{g}^{z\bar{z}} \kappa_{\bar{z} z}).
\end{aligned}
\end{equation}
We use integration by parts on the first term on the right-hand side to obtain
\begin{equation}
\begin{aligned}
{1 \over p} {d \over d t} \int_\Sigma e^{pf} &= -p \int_\Sigma e^{pf} |\p f|^2 +  {\alpha' \over 2} (p-2) \int_\Sigma e^{(p-2)f} \kappa |\p f|^2 + {\alpha' \over 2} \int_\Sigma e^{(p-2)f} \p \kappa \wedge \bar{\p} f  \nonumber\\
& + \int_\Sigma e^{pf} (1 + {\alpha' \over 2} \kappa e^{-2f}) |\p f|^2  - \alpha' \int_\Sigma e^{pf} e^{-2f} {\rm Re} \{ \hat{g}^{z \bar{z}} \p_z \kappa \p_{\bar{z}} f \} \nonumber\\
& -  \int_\Sigma e^{pf} \kappa  - {\alpha' \over 2}  \int_\Sigma e^{pf} e^{-2f} (\kappa^2 - \hat{g}^{z\bar{z}} \kappa_{\bar{z} z}).
\end{aligned}
\end{equation}
Simplifying the above expression, and using $a^{z \bar{z}} = (1- \frac{\alpha'}{2} \kappa e^{-2f}) \hat{g}^{z \bar{z}}$ yields
\[
{1 \over p} {d \over d t} \int_\Sigma e^{pf} =  -(p-1) \int_\Sigma e^{pf} a^{z \bar{z}} f_z f_{\bar{z}} - {\alpha' \over 2}  \int_\Sigma e^{(p-2)f}\p f \wedge \bar{\p} \kappa  -  {\alpha' \over 2} \int_\Sigma e^{(p-2)f}(\kappa^2 - \hat{g}^{z\bar{z}} \kappa_{\bar{z} z})  -  \int_\Sigma e^{pf} \kappa.
\]
Using integration by parts again gives
\[
(p-1) \int_\Sigma e^{pf} a^{z \bar{z}} f_z f_{\bar{z}} + {1 \over p} {d \over d t} \int_\Sigma e^{pf}
= {\alpha' \over 2} \bigg(1 + {1 \over p-2}\bigg) \int_\Sigma e^{(p-2)f} \hat{g}^{z\bar{z}} \kappa_{\bar{z} z}  -  \int_\Sigma e^{pf} \kappa \, (1 + {\alpha' \over 2} \kappa e^{-2f}) .
\]
For any $p \geq 3$, we estimate
\[
(p-1) \int_\Sigma e^{pf} |\p f|^2 + {d \over d t}{1 \over p} \int_\Sigma e^{pf} \leq C \left( \int_\Sigma e^{pf} + \int_\Sigma e^{(p-2)f} \right).
\]
Here we used $a^{z \bar{z}} \geq \hat{g}^{z \bar{z}}$. In the proof of this proposition, we let $C$ denote any constant depending only on $(\Sigma,\hat{\omega})$, $\kappa$, and $\alpha'$, which may change line by line.
\smallskip
\par For $0< \tau < \tau' < T$, we introduce $\zeta(t) \geq 0$ which is a monotone function satisfying $\zeta \equiv 0$ for $t \leq \tau$, $\zeta \equiv 1$ for $t \geq \tau'$, and $|\zeta'| \leq 2(\tau'-\tau)^{-1}$. For any $p \geq 3$, we have the inequality
\[
{p \zeta \over 2} \int_\Sigma  e^{pf} |\p f|^2 + {d \over d t} {\zeta \over p} \int_\Sigma  e^{pf}
\leq C \bigg\{ \zeta \int_\Sigma e^{(p-2)f} +  \zeta \int_\Sigma  e^{pf} \bigg\} + {\zeta' \over p} \int_\Sigma  e^{pf} .
\]
Let $\tau' < s \leq T$, and integrate the previous inequality from $\tau$ to $s$.
\[ {p \over 2} \int_{\tau'}^{s} \int_\Sigma  e^{pf} |\p f|^2 + {1 \over p} \int_\Sigma  e^{pf}(s) \leq C \bigg\{ \int_{\tau}^T\int_\Sigma e^{(p-2)f} +  \int_{\tau}^T\int_\Sigma  e^{pf} + {1 \over \tau' - \tau} \int_{\tau}^T \int_\Sigma  e^{pf} \bigg\}. \]
We conclude the estimate
\[
 \int_{\tau'}^{s} \int_\Sigma  |\p e^{{p \over 2} f}|^2 +  \int_\Sigma  e^{p f}(s) \leq C p \bigg\{ 1 +  {1 \over \tau' - \tau} \bigg\} \bigg\{ \int_{\tau}^T \int_\Sigma  e^{p f} +\int_{\tau}^T\int_\Sigma e^{(p-2)f} \bigg\},
\]
which after using H\"older's inequality leads to
\be \label{iteration_sup_est}
 \int_{\tau'}^{s} \int_\Sigma  |\p e^{{p \over 2} f}|^2 +  \int_\Sigma  e^{p f}(s) \leq C p \bigg\{ 1 +  {1 \over \tau' - \tau} \bigg\} \bigg\{ \int_{\tau}^T \int_\Sigma  e^{p f} + T^{2/p} \bigg( \int_{\tau}^T \int_\Sigma  e^{p f} \bigg)^{p-2 \over p}\bigg\}.
\ee
Recall that we are working on a manifold $(\Sigma,\hat{\omega})$ of real dimension $n=2$. By the Sobolev inequality 
\[
 \left( \int_\Sigma  e^{3p f} \right)^{1 \over 6} \leq C  \left( \int_\Sigma | e^{{p \over 2}f}|^{3/2}  + \int_\Sigma |\p e^{{p \over 2}f}|^{3/2} \right)^{2/3} .
\]
By H\"older's inequality
\[
 \left( \int_\Sigma  e^{3 p f} \right)^{1 \over 3} \leq C  \left( \int_\Sigma | e^{{p \over 2}f}|^{2}  + \int_\Sigma |\p e^{{p \over 2}f}|^{2} \right) .
\]
Let $\beta=3$ and $\beta^*={3 \over 2}$, and note that ${1 \over \beta} + {1 \over \beta^*} = 1$. The previous inequality implies
\bea
\int_{\tau'}^{T} \int_\Sigma e^{p f} e^{{p \over \beta^*} f} &\leq& \int_{\tau'}^{T} \bigg(  \int_\Sigma e^{p \beta f} \bigg)^{1/\beta} \bigg( \int_\Sigma e^{pf} \bigg)^{1/\beta^*} \nonumber\\
&\leq& C \sup_{s \in [\tau',T]} \bigg( \int_\Sigma e^{pf} \bigg)^{1/\beta^*} \int_{\tau'}^{T} \bigg\{ \int_\Sigma e^{pf} + \int_\Sigma |\p e^{{p \over 2}f}|^2 \bigg\}. \nonumber
\eea
Let $\gamma = 1+{1 \over \beta^*} = {5 \over 3}$. Using estimate (\ref{iteration_sup_est}),
\[
\bigg( \int_{\tau'}^{T} \int_\Sigma  e^{ \gamma p f} \bigg)^{1/ \gamma} \leq  C (1+T^{2/p}) p \bigg\{ 1+  {1 \over \tau' - \tau} \bigg\}  \max \bigg\{ 1, \int_{\tau}^T \int_\Sigma  e^{p f} \bigg\}.
\]
for $p \geq 3$. Let $0 < \theta_2 < \theta_1 < T$ and $\tau_p = \theta_1 - \gamma^{-p} (\theta_1-\theta_2)$. Replacing $p$ with $\gamma^p$ and setting $\tau'=\tau_{p+1}$ and $\tau=\tau_p$ gives
\[
\bigg( \int_{\tau_{p+1}}^{T} \int_\Sigma  e^{ \gamma^{p+1} f} \bigg)^{1/ \gamma^{p+1}} \leq  (C(1+T^{2/\gamma^p}))^{1/\gamma^p} \bigg\{ \gamma^p +  {1 \over \theta_1 -\theta_2} {\gamma^{2p} \over 1 - \gamma^{-1}} \bigg\}^{1 / \gamma^p} \max \bigg\{ 1, \int_{\tau_p}^T \int_\Sigma  e^{\gamma^p f} \bigg\}^{1 / \gamma^p}.
\]
We now iterate this inequality down to $5 \geq \gamma^{3} \geq 3$. Noting that $\sum \gamma^{-p} \leq 3$, we obtain
\[
\sup_{\Sigma \times [\theta_1,T]} e^{f} \leq {C(T) \over (\theta_1-\theta_2)^3 } \max\{ 1, \| e^{f} \|_{L^{5}(\Sigma \times [\theta_2,T])} \},
\]
Here $C(T)$ only depends on $(\Sigma,\hat{\omega})$, $\kappa$, $T$ and $\alpha'$. We must now relate the $L^{5}$ norm of $e^f$ to its $L^1$ norm, which we can do via a scaling argument such as in \cite{han2011}. By Young's inequality,
\bea
\sup_{\Sigma \times [\theta_1,T]} e^{f} &\leq&  C(T) (\theta_1-\theta_2)^{-3} \max \left\{ \bigg( \sup_{\Sigma \times [\theta_2,T]} e^{(1 -1/5)f} \bigg) \bigg( \int_{\Sigma \times [\theta_2,T] } e^{f} \bigg)^{1/5}, 1 \right\} \nonumber\\
&\leq& {1 \over 2}  \sup_{\Sigma \times [\theta_2,T]} e^f + C(T) (\theta_1-\theta_2)^{-15} \max \left\{1, \int_{\Sigma \times [0,T] } e^{f} \right\} \nonumber,
\eea
for all $0<\theta_2<\theta_1$. Now let $0<T_0<T$ and $p>1$. Iterating this inequality with $\theta_0=T_0$ and $\theta_{i+1} = \theta_i - {1 \over 2}(1-r)r^{i+1} T_0 $, where $r$ is chosen such that $1/2 < r^{15} < 1$, we obtain
\[
\sup_{\Sigma \times [T_0,T]} e^{f} \leq {1 \over 2^p} \left( \sup_{\Sigma \times[\theta_p,T]} e^f \right) + {2^{15} \, C(T) \over (1-r)^{15} r^{15} T_0^{15}} \sum_{i=0}^{p-1} \left( {1 \over 2 r^{15}} \right)^{i} \max \left\{1, \int_{\Sigma \times [0,T] } e^{f} \right\}.
\]
Taking the limit as $p \rightarrow \infty$, we note that $\theta_p \rightarrow (1-\frac{r}{2})T_0$ and the first term goes to zero. This yields the desired estimate.

\end{proof}

\section{Large initial data}
\label{collapsing-section}

There is a class of initial data where the flow (\ref{flow-ef}) on a vanishing spinorial pair $(\Sigma,\varphi)$ exists for all time. Since $- \kappa \geq 0$, an application of the maximum principle to (\ref{flow-u-2}) shows that the condition $u \geq 0$ is preserved along the flow. In terms of $f$, this means
\be \label{u-geq-0}
e^{2f} \geq {\alpha' \over 2} (-\kappa).
\ee
Solutions in this region will be said to have large initial data, and in this section we will analyse these solutions. We recall the convention that norms and integrals are taken with respect to the background metric $\hat{\omega}$.
\begin{thm} \label{u>0-conv}
Suppose $u(x,0) \geq 0$, or equivalently (\ref{u-geq-0}), and start the Anomaly flow $(\ref{flow-ef})$. Then the flow exists for all time, and as $t \rightarrow \infty$,
\[
{e^f \over \left( \int_\Sigma e^{2f} \right)^{1/2}} \rightarrow q_1
\]
smoothly, where $q_1$ is the first eigenfunction of the operator $-\Delta_{\hat{\omega}} + 2 \kappa$ with normalization $q_1>0$ and $\|q_1\|_{L^2(\Sigma,\hat{\omega})}=1$.
\end{thm}

\begin{rmk}
This result implies Theorem \ref{spinorial-pair-flow-thm} by letting $\omega(t) = e^{f(t)} \hat{\omega}$.
\end{rmk}

\par We note that if $u(x,0) \geq 0$ at the initial time, then by the strong maximum principle, for $t>0$ we have $u(x,t)>0$. Indeed, let $B \gg 1$ be such that $2(-\kappa) - B \leq 0$. Let $u_{B} =e^{-B t} u$. Then using the evolution of $u$ (\ref{flow-u}) we obtain the evolution of $u_{B}$:
\[
\p_t u_{B} - a^{z \bar{z}} \p_z \p_{\bar{z}} u_{B} - \left( - B + (-\kappa) (1 -  {\alpha' \over 2} \kappa e^{-2f}) \right) u_{B} =0.
\]
By (\ref{u-geq-0}) and choice of $B$, we have
\[
\bigg( - B + (-\kappa) (1 -  {\alpha' \over 2} \kappa e^{-2f}) \bigg) \leq 0.
\]
Therefore we may apply the strong maximum principle \cite{nirenberg1953} to conclude either $u_B >0$ for all $t>0$ or $u_B \equiv 0$. But $u$ cannot be identically zero by its definition, since at a branch point $p$ of $\varphi$ we have $\kappa(p) =0$ and $u(p) = e^f(p)$. This implies $u>0$ for all $t>0$.
\medskip
\par Therefore, after only considering times greater than a fixed small time $t_0>0$, we may assume that $u>2\delta$ along the flow, which means in terms of $f$ that
\be \label{delta-big}
e^{f} > \sqrt{ {\alpha' \over 2} (-\kappa)} + \delta.
\ee
for some $\delta>0$. This provides a uniform upper bound for $e^{-f}$, and we can apply the long-time existence criterion (Theorem \ref{long-time-criterion}) to conclude that the flow exists for all time $t \in [0,\infty)$.
\smallskip
\par Though we now have a solution for all time $t \in [0,\infty)$, we will obtain more refined estimates to understand its behavior at infinity. In the following sections, we use the standard convention that constants $C$ depending on known quantities may change line by line.

\subsection{Integral growth}
\par Let $q_1$ be the first eigenfunction of the operator $-\Delta_{\hat{\omega}} + 2\kappa$ with eigenvalue $\lambda_1$. It is well-known that $q_1 > 0$ and $\lambda_1 < 0$. To avoid sign confusion, we let $0< \eta = - {\lambda_1 \over 2}$. Our first estimate concerns the exponential growth of the integral $\int_\Sigma e^f$.

\begin{prop}
Let $\delta>0$, and start the flow with $u(x,0) > 2 \delta$. Then there exists a constant $C>1$ depending on $(\Sigma,\varphi)$, $\alpha'$ and $\delta$ such that
\be \label{int-e^f-growth}
C^{-1} e^{\eta t} \leq \int_\Sigma e^f \leq C e^{\eta t}.
\ee
\end{prop}

\begin{proof}
We first compute the evolution of the inner product of $e^f$ with $q_1$.
\bea
{d \over dt} \int_\Sigma (e^f q_1) \, \hat{\omega} &=& \int_\Sigma  q_1 \, i \ddb (e^f + {\alpha' \over 2} \kappa e^{-f})- \int_\Sigma q_1 \kappa (e^f + {\alpha' \over 2} \kappa e^{-f}) \hat{\omega} \nonumber\\
&=& \int_\Sigma (e^f + {\alpha' \over 2} \kappa e^{-f}) (i \ddb q_1 - \kappa q_1 \hat{\omega}) \nonumber\\
&=& \eta \int_\Sigma  q_1 (e^f + {\alpha' \over 2} \kappa e^{-f}) \hat{\omega}. \nonumber
\eea
We will often omit the volume form $\hat{\omega}$ when integrating. Since $q_1 \geq 0$ and $\kappa \leq 0$, we have
\[ {d \over dt} \int_\Sigma e^f q_1 \leq \eta \int_\Sigma e^f q_1. \]
Therefore
\[
\int_\Sigma e^f q_1 \leq C e^{\eta t}.
\]
On the other hand, by (\ref{u-geq-0}) we have
\[ {d \over dt} \int_\Sigma e^f q_1 \geq \eta \int_\Sigma q_1 e^f - \eta \int_\Sigma q_1 \sqrt{ {\alpha' |\kappa| \over 2} }  . \]
It follows that
\[
{d \over dt} \bigg( e^{-\eta t} \int_\Sigma e^f q_1 - e^{-\eta t} \int_\Sigma q_1 \sqrt{ {\alpha' |\kappa| \over 2} } \bigg) \geq 0,
\]
and integrating this differential inequality gives
\[
 \int_\Sigma e^f q_1 \geq \bigg( \int_\Sigma e^f(0) q_1 - \int_\Sigma q_1 \sqrt{ {\alpha' |\kappa| \over 2} }\bigg) e^{\eta t} +\int_\Sigma q_1 \sqrt{ {\alpha' |\kappa| \over 2} }.
\]
Using (\ref{delta-big}), we have
\[
 \int_\Sigma e^f q_1 \geq \delta \left( \int_\Sigma q_1 \right) e^{\eta t} +\int_\Sigma q_1 \sqrt{ {\alpha' |\kappa| \over 2} }.
\]
Combining both bounds on $\int e^f q_1$ gives
\[C^{-1} e^{\eta t} \leq \int_\Sigma e^f q_1 \leq C e^{\eta t}.\]
Since $q_1>0$ on $\Sigma$, we obtain the desired estimate.
\end{proof}

\subsection{Estimates}
In this section, we obtain more precise estimates for $u$ as $t \rightarrow \infty$.
\begin{prop} \label{int-u-comp}
Suppose $u> 2 \delta$ at $t=0$. There exists $T>0$ and $C>1$ depending on $(\Sigma,\varphi)$, $\alpha'$ and $\delta$ with the following property. For all $t_1,t_2 \geq T$ such that $|t_1-t_2| \leq 1$, then
\[ C^{-1} \int_\Sigma u(t_2) \leq \int_\Sigma u(t_1) \leq C \int_\Sigma u(t_2). \]
\end{prop}

\begin{proof}
Since $u = e^f + \frac{\alpha'}{2} e^{-f}\kappa$, by the growth of $\int_\Sigma e^f$ (\ref{int-e^f-growth}) and the upper bound of $e^{-f}$ (\ref{delta-big}), we have
\[ C^{-1} e^{\eta t} - C \leq \int_\Sigma u \leq C e^{\eta t}, \]
for all $t \in [0,\infty)$. It follows that there exists $T>0$ such that for all $t \geq T$, then
\be \label{int-u-growth}
{C^{-1} \over 2} e^{\eta t} \leq \int_\Sigma u \leq C e^{\eta t}.
\ee
The desired estimate follows.
\end{proof}

\begin{prop} \label{u-vs-L2}
Start the flow with $u(x,0) > 2 \delta$. Then there exists $T>0$ and $C>1$ depending on $(\Sigma,\varphi)$, $\alpha'$ and $\delta$ such that
\[C^{-1} \left( \int_\Sigma u^2 \right)^{1/2} \leq u(x,t) \leq C \left( \int_\Sigma u^2 \right)^{1/2}, \]
for all $t \geq T$.
\end{prop}

\begin{proof}
Fix $t_0 \in (T,\infty)$, where $T$ is as in Proposition \ref{int-u-comp}. For the following arguments, we will assume that $T \gg 1$. Let $n$ be a real number such that $t_0 \in [n+\frac{1}{2},n+1]$. As before, we have
\[ (\p_t - a^{z \bar{z}} \p_z \p_{\bar{z}}) \, e^{-B (t-t_0)} u  \leq 0,\]
for $B \geq 2 \sup_\Sigma |\kappa|$ and $\hat{g}^{z \bar{z}} \leq a^{z \bar{z}} \leq \Lambda \hat{g}^{z \bar{z}}$. By the local maximum principle \cite[Theorem 7.36]{lieberman1996}, for every $p>0$ there exists a uniform $C>0$ such that in a local coordinate ball $B_1$ there holds
\be \label{supu-vs-intu}
\sup_{B_{1/2} \times [n+\frac{1}{2},n+1]} e^{-B (t-t_0)} u \leq C \bigg( \int_{n}^{n+1} \int_{B_1} (e^{-B (t-t_0)} u)^p \bigg)^{1/p}.
\ee
Let us take $p=1$, and center this coordinate chart around a point $p \in \Sigma$ where $u(x,t_0)$ attains its maximum. Since $\int_\Sigma u$ is comparable at all nearby times by Proposition \ref{int-u-comp},
\[ \sup_{\Sigma} u(t_0) \leq C \int_\Sigma u (t_0). \]
It follows that for all $t>T$, then
\[C^{-1}\| u \|_{L^1(\Sigma)}(t) \leq \| u \|_{L^2(\Sigma)}(t) \leq C \| u \|_{L^1(\Sigma)} (t).\]
Hence by Proposition \ref{int-u-comp}, $\| u \|_{L^2(\Sigma)}$ is also comparable at all nearby times. Stated explicitly, for $t_1,t_2 \geq T$ and $|t_2-t_1| \leq 1$, then
\be \label{comp-L2}
C^{-1} \| u \|_{L^2(\Sigma)} (t_2) \leq \| u \|_{L^2(\Sigma)} (t_1) \leq C \| u \|_{L^2(\Sigma)} (t_2).
\ee
Next, choosing $t_0 \in (T,\infty)$ and $t_0 \in [n,n+1]$, we observe
\[ (\p_t - a^{z \bar{z}} \p_z \p_{\bar{z}}) \, e^{B (t-t_0)} u \geq 0.\]
Cover $\Sigma$ with finitely many local coordinate balls $U_i$. By the weak Harnack inequality \cite[Theorem 7.37]{lieberman1996}, for some $p>0$ there holds
\[ \inf_{U_i \times [n,n+1]} e^{B (t-t_0)} u \geq C^{-1} \bigg( \int_{n-2}^{n-1} \int_{U_i} (e^{B (t-t_0)}u)^p \bigg)^{1/p}. \]
Suppose the infimum of $e^{B (t-t_0)} u$ on $\Sigma \times [n,n+1]$ is attained in $U_1$. Let $U_2$ be another chart such that $U_1 \cap U_2 \neq \emptyset$. Then
\bea
\bigg( \int_{n-4}^{n-3} \int_{U_2} (e^{B (t-t_0)}u)^p \bigg)^{1/p} &\leq& C \inf_{U_2 \times [n-2,n-1]} e^{B (t-t_0)} u \nonumber\\
&\leq&  C \inf_{U_2 \cap U_1 \times [n-2,n-1]} e^{B (t-t_0)} u \nonumber\\
&\leq& C \bigg( \int_{n-2}^{n-1} \int_{U_1} (e^{B (t-t_0)}u)^p \bigg)^{1/p} \nonumber\\
&\leq& C \inf_{\Sigma \times [n,n+1]} e^{B (t-t_0)} u. \nonumber
\eea
There exists a uniform $m_0>0$ depending on the covering $\Sigma \subseteq \bigcup U_i$ such that after applying this argument $m_0$ times, we can deduce
\[ \inf_{\Sigma \times [n,n+1]} e^{B (t-t_0)} u \geq C^{-1} \bigg( \int_{n-(m_0+1)}^{n-m_0} \int_\Sigma (e^{B (t-t_0)}u)^p \bigg)^{1/p}. \]
We can assume that $k_0=n -m_0 > 1$ since $T \gg 1$. By (\ref{supu-vs-intu}), we obtain
\[ \bigg( \int_{k_0-1}^{k_0} \int_\Sigma (e^{B (t-t_0)}u)^p \bigg)^{1/p} \geq C^{-1} \sup_{\Sigma \times [k_0-1/2,k_0]} u \geq C^{-1} \bigg( \int_{k_0-1/2}^{k_0} \int_\Sigma u^2 \bigg)^{1/2}.\]
Combining these estimates
\[ \inf_{\Sigma} u(t_0) \geq \inf_{\Sigma \times [n,n+1]} u \geq C^{-1} \bigg( \int_{n-m_0-1/2}^{n-m_0} \int_\Sigma u^2 \bigg)^{1/2}. \]
By (\ref{comp-L2}), we see that $\int_\Sigma u^2$ is comparable at all times in a bounded interval, hence
\[ \inf_{\Sigma} u(t_0) \geq C^{-1} \| u \|_{L^2(\Sigma)} (t_0). \]

\end{proof}

We now introduce the normalized function
\be \label{v-defn}
v(x,t) = {u(x,t) \over \| u \|_{L^2(\Sigma)} (t)}.
\ee
We have established that for $t \geq T$,
\[ C^{-1} \leq v(x,t) \leq C.\]
We now obtain uniform higher order estimates for $v$.
\begin{prop} \label{higher-est-v}
Suppose $u>2 \delta$ at $t=0$. There exists $T>0$ depending on $(\Sigma,\varphi)$, $\alpha'$ and $\delta$ with the following property. For each $k$, there exists $C_k>0$ depending on $(\Sigma,\varphi)$, $\alpha'$ and $\delta$ such that the normalized function $v = u / \| u \|_{L^2(\Sigma)}$ can be estimated by
\[ \| v \|_{C^k(\Sigma)}(t) \leq C_k, \]
for any $t \in (T,\infty)$.
\end{prop}
\begin{proof}
Let $T$ be as in the proof of Proposition \ref{u-vs-L2}. We fix $t_0 \in (T,\infty)$, $t_0 \in [n,n+1]$ as before and consider
\[ w = {u(x,t) \over \| u \|_{L^2(\Sigma)}(t_0)}. \]
By (\ref{comp-L2}), we have the estimate
\[ C^{-1} \leq w(x,t) \leq C \]
for $t \in [n,n+1]$ and $w$ satisfies
\[ \p_t w - a^{z \bar{z}} w_{\bar{z} z} + (\kappa a^{z \bar{z}} \hat{g}_{\bar{z} z}) w = 0, \ \ \hat{g}^{z \bar{z}} \leq a^{z \bar{z}} \leq \Lambda \hat{g}^{z \bar{z}}, \ \  0 \leq (-\kappa a^{z \bar{z}} \hat{g}_{\bar{z} z}) \leq \Lambda. \]
By the Krylov-Safonov theorem \cite{krylov1980,krylov1981}, there exists $\delta>0$ such that
\[ \| w \|_{C^{\delta,\delta/2}(\Sigma \times [n,n+1])} \leq C. \]
The H\"older norm of $e^f = \frac{1}{2}(u+\sqrt{u^2-2\alpha'\kappa})$ on $\Sigma \times [n,n+1]$ can now be estimated by a constant times $ \| u \|_{L^2(\Sigma)}(t_0)$. For $x,y$ in the same coordinate chart and $t,s \in [n,n+1]$, we have
\[ |e^{-f(x,t)} - e^{-f(y,s)}| \leq {C \| u \|_{L^2(\Sigma)}(t_0) (|x-y|+|t-s|^{1/2})^\delta \over e^{f(x,t)} e^{f(y,s)}} \leq 2 C  {1 \over  e^{f(x,t)}} {\| u \|_{L^2(\Sigma)}(t_0) \over u(y,s)}(|x-y|+|t-s|^{1/2})^\delta . \]
Thus we have $\| e^{-f} \|_{ C^{\delta,\delta/2} (\Sigma \times [n,n+1]) } \leq C$. This implies a H\"older estimate for $a^{z \bar{z}}$, and we may apply Schauder estimates \cite{krylov1996} to bound $w$ uniformly in $C^{2+\delta,1+\delta/2}(\Sigma \times [n,n+1])$. Higher order estimates follow by a bootstrap argument.
\smallskip
\par We have obtained estimates on spacial derivatives of $u$ on the time interval $[n,n+1]$ in terms of $\| u \|_{L^2(\Sigma)}(t_0)$. By (\ref{comp-L2}), it follows that $\| v \|_{C^k(\Sigma)}(t) \leq C_k$ uniformly.
\end{proof}

Our last estimate concerns the function $f$, and is a consequence of our work so far.
\begin{prop}\label{smallcurv}
Suppose $u>2 \delta$ at $t=0$. There exists $T>0$ depending on $(\Sigma,\varphi)$, $\alpha'$ and $\delta$ with the following property. For each integer $k$, there exists $C_k>0$ depending on $(\Sigma,\varphi)$, $\alpha'$ and $\delta$ such that on $(T,\infty)$,
\be \label{exp(-f)-decay}
 e^{-f} \leq C_0 e^{-\eta t}, \ \ \ \| \nabla^k f \|_{L^\infty(\Sigma \times (T,\infty))} \leq C_k \ \mbox{for $k \geq 1$}.
\ee
\end{prop}
\begin{proof}
Since $u \leq e^f$, by Proposition \ref{u-vs-L2} we know
\[ e^{-f} \leq {C \over \| u \|_{L^2(\Sigma)}} \leq {C \over \int_\Sigma u}. \]
By (\ref{int-u-growth}), for all $t \geq T$, we have $e^{-f} \leq C e^{-\eta t}$. Next, by the definition of $u$ in terms of $f$, we note the identity
\[ u \p_z f = \p_z u - {\alpha' \over 2} e^{-f} \p_z \kappa. \]
Combining Proposition \ref{u-vs-L2} and Proposition \ref{higher-est-v}, we have a uniform bound for $\dfrac{\p_z u}{u}$, and a lower bound for $u$. It follows that $\p_z f$ is uniformly bounded. Further differentiating the identity above gives higher order estimates of $f$.
\end{proof}

\subsection{Convergence}

With the estimates obtained in the previous section, we can now show convergence of a normalization of $e^f$ along the flow, for initial data satisfying $u(x,0) > 2\delta$.
\smallskip
\par From the definition of $v$ (\ref{v-defn}) and the evolution of $u$ (\ref{flow-u}), we have the following evolution equation
\bea \label{flow-v}
\p_t v &=& (1 - {\alpha' \over 2} \kappa e^{-2f}) \hat{g}^{z \bar{z}} v_{\bar{z} z} - \kappa (1 - {\alpha' \over 2} \kappa e^{-2f}) v \nonumber\\
&&- v \int_\Sigma v \bigg( 1 - {\alpha' \over 2} \kappa e^{-2f} \bigg)  (\hat{g}^{z \bar{z}} v_{\bar{z} z} - \kappa v).
\eea
We will look at the energy of $v$ along the flow.
\[ I(v) = {1 \over 2} \int_\Sigma |\p v|^2  + {1 \over 2} \int_\Sigma \kappa v^2. \]
Differentiating gives
\bea
{d \over dt} I (v) &=& - \int_\Sigma \dot{v} \hat{g}^{z \bar{z}} v_{\bar{z} z} + \int_\Sigma \kappa v \dot{v} \nonumber\\
&=& - \int_\Sigma \dot{v} \bigg(\dot{v} + {\alpha' \over 2} \kappa e^{-2f} \hat{g}^{z \bar{z}} v_{\bar{z} z} + \kappa (1 - {\alpha' \over 2} \kappa e^{-2f}) v \bigg) \nonumber\\
&& - \int_\Sigma \dot{v} v  \int_\Sigma v \bigg( 1 - {\alpha' \over 2} \kappa e^{-2f} \bigg)  (\hat{g}^{z \bar{z}} v_{\bar{z} z} - \kappa v)+ \int_\Sigma \kappa v \dot{v} \nonumber.
\eea
From differentiating $\int_\Sigma v^2 = 1$, we see that $\int_\Sigma v \dot{v} = 0$. Therefore
\[
{d \over dt} I(v) = - \int_\Sigma \dot{v}^2 - {\alpha' \over 2} \int_\Sigma  ( \kappa e^{-2f}) (\hat{g}^{z \bar{z}} v_{\bar{z} z} - \kappa v) \, \dot{v}.
\]
By Proposition \ref{higher-est-v}, we have $\| v \|_{C^2(\Sigma)}(t) \leq C$ along the flow. By (\ref{flow-v}), we see that $\dot{v}$ is also uniformly bounded along the flow. By (\ref{exp(-f)-decay}), it follows that there exists $T>0$ such that for all $t \geq T$ then
\be \label{I-dot-est}
{d \over dt} I(v) \leq - \int_\Sigma \dot{v}^2 + C \sup_\Sigma e^{-2f} \leq  - \int_\Sigma \dot{v}^2 + C e^{-\eta t}.
\ee
We claim that as $t \rightarrow \infty$, we have that $\int_\Sigma \dot{v}^2 \rightarrow 0$. Suppose this is not the case. Then there exists a sequence $t_n \rightarrow \infty$ such that $\int_\Sigma \dot{v}^2(t_n) \geq \epsilon>0$. By our estimates,
\[ \bigg| {d \over dt} \int_\Sigma \dot{v}^2 \bigg| \leq C, \]
therefore there exists $\delta>0$ such that $\int_\Sigma \dot{v}^2 \geq \epsilon/2$ on $[t_n -\delta, t_n + \delta]$. Using (\ref{I-dot-est}), we obtain
\[ I(v)(s) - I(v)(T) = \int_T^s {d \over dt} I(v) d t \leq - \int_T^s \int_\Sigma \dot{v}^2 + C \int_T^s e^{-\eta t} dt, \]
and we see that $I(v)(s)$ is not bounded below as $s \rightarrow \infty$, which is a contradiction.
\smallskip
\par We can now show that $v$ converges smoothly to $q_1$, the first eigenfunction of the operator $-\Delta_{\hat{\omega}} + 2\kappa$. Indeed, suppose this does not hold. Then there exists a sequence of $t_i \rightarrow \infty$ such that after passing to a subsequence we have $v \rightarrow v_\infty$ smoothly and $v_\infty \neq q_1$. Applying Proposition \ref{higher-est-v} to the expression for $\dot{v}$ (\ref{flow-v}), we may use the Arzela-Ascoli theorem and assume that $\dot{v}(t_i)$ converges uniformly to some function. Since $\int_\Sigma \dot{v}^2 \rightarrow 0$, we conclude that $\dot{v}(t_i) \rightarrow 0$. Letting $t_i \rightarrow \infty$ in the evolution equation of $v$ (\ref{flow-v}), we see that
\[ \hat{g}^{z \bar{z}} (v_\infty)_{\bar{z} z} - \kappa v_\infty = \eta \, v_\infty, \ \ \eta = - \, {\int_\Sigma \kappa v_\infty \over \int_\Sigma v_\infty}, \]
with
\[ v_\infty > 0, \ \ \| v_\infty \|_{L^2(\Sigma)} = 1.\]
This identifies $v_\infty$ as $q_1$, a contradiction.
\smallskip
\par To complete the proof of Theorem \ref{u>0-conv}, we remark
\[\frac{\|u \|_{L^2(\Sigma)}}{\| e^f \|_{L^2(\Sigma)}}\to 1\]
and
\[
{e^f \over \| e^f \|_{L^2(\Sigma)}} = v\cdot\frac{\| u \|_{L^2(\Sigma)}}{\| e^f \|_{L^2(\Sigma)}} - {\alpha' \over 2} {e^{-f} \over \| e^f \|_{L^2(\Sigma)}} \kappa \rightarrow v_\infty.
\]

\subsection{Collapsing of the hyperk\"ahler fibers}

In the previous section, we gave the proof of Theorem \ref{u>0-conv}. We would like to interpret this theorem geometrically. On the threefold $X$, we are studying the evolution of the metric $\omega_f = e^{2f} \hat{\omega} + e^f \omega'$ under the Anomaly flow. By (\ref{volume-omega_f}), if we assume $\int_{T^4} d {\rm vol}_{T^4} = 1$, then
\[
{\omega_f \over {1 \over 3!} \int_X \| \Omega \|_{\omega_f} \, \omega_f^3} = \bigg( {e^f \over \| e^f \|_{L^2(\Sigma)}} \bigg)^2 \hat{\omega} + \bigg( {e^f \over \| e^f \|_{L^2(\Sigma)}} \bigg) {1 \over \| e^f \|_{L^2(\Sigma)}} \omega'.
\]
We see that if $u(x,0) \geq 0$, then as $t \rightarrow \infty$ the hyperk\"ahler fibers are collapsing and the rescaled metrics converge to the following metric on the base
\[
{\omega_f \over \frac{1}{3!} \int_X \| \Omega \|_{\omega_f} \, \omega_f^3} \rightarrow q_1^2 \, \hat{\omega}.
\]
From here, it can be established that $(X,{\omega_f \over \frac{1}{3!} \int_X \| \Omega \|_{\omega_f} \, \omega_f^3})$ converges to $(\Sigma,q_1^2 \, \hat{\omega})$ in the Gromov-Hausdorff sense; this statement can be found in (\cite[Theorem 5.23]{tosatti2015c}).
\medskip
\par The Anomaly flow has produced a limiting metric $\omega_\Sigma = q_1^2 \, \hat{\omega}$ which can be associated to a vanishing spinorial pair $(\Sigma,\varphi)$. Its curvature is given by
\[
- i \ddb \log \omega_\Sigma = - (2 \eta q_1^{-2}) \omega_\Sigma + \varphi^* \omega_{FS} + 2 i q_1^{-2} \p q_1 \wedge \bar{\p} q_1.
\]
In section \S \ref{gcg-setup}, we showed that $\varphi^* \omega_{FS}= \omega_{WP}$.

\subsection{Small curvature condition}

It was shown in \cite{phong2017} that the Anomaly flow exists for a short-time if $|\alpha' Rm(\omega_0)|$ is small initially. In this subsection, we show that under the reduction of the Anomaly to the Riemann surface, our long-time existence result (Theorem \ref{u>0-conv}) can be interpreted as the condition
\[|\alpha'Rm(\omega_f)| \ll 1\]
being preserved under the flow (\ref{flow-ef}).

The first step is to compute $|Rm(\omega_f)|$ in terms of $f$. Based on the complicated calculation in \cite{fei2016, fei2015d}, one can compute directly that
\[|Rm(\omega_f)| \sim e^{-2f}+e^{-2f}|\pt f|+e^{-2f}|\Delta_{\hat{\omega}} f|.\]

It follows that if $|\alpha'Rm(\omega_f)| \ll 1$ initially, then
\[u=e^f\left(1+\frac{\alpha'e^{-2f}}{2}\kappa\right)\geq0\]
initially, hence we have long-time existence. Moreover by Proposition \ref{smallcurv}, we deduce that the condition $|\alpha'Rm(\omega_f)| \ll 1$ is ultimately preserved under the flow and in fact this quantity decays exponentially. Hence we have proved Theorem \ref{af_collapse_thm}.

In \cite{phong2016c}, it is shown that the Anomaly flow with $\alpha'=0$ exists as long as $|Rm|^2+|DT|^2+|T|^4$ remains bounded. Here $T$ is the torsion tensor associated to the Chern connection. For our reduced flow (\ref{flow-ef}) on Riemann surfaces with $\alpha'>0$, a similar calculation indicates that \[|Rm|^2+|DT|^2+|T|^4 \sim e^{-4f}+e^{-4f}|\pt f|^4+e^{-4f}|\nabla^2 f|^2.\]
If this quantity is bounded, then in particular $e^{-f}$ remains bounded, and by Theorem \ref{long-time-criterion} the flow can be extended. This observation suggests the possibility of generalizing the long-time existence criterion in \cite{phong2016c} to the case when $\alpha'>0$.




\section{Further directions} \label{future-dir}


In this section, we will specialize to Equation (\ref{flow-ef})
\[\p_t e^f =  \hat{g}^{z \bar{z}} \p_z \p_{\bar{z}} (e^f + {\alpha' \over 2}  \kappa e^{-f} ) - \kappa  (e^f + {\alpha' \over 2} \kappa e^{-f} ),\]
which comes from the reduction of the Anomaly flow to a Riemann surface. From previous sections, we see that the behavior of this parabolic equation is very sensitive to the initial data. Indeed, for large initial data, we have the long-time existence of solutions, however the flow does not converge without normalization. On the other hand, the flow will develop a finite time singularity if the initial data is small in the $L^1$-sense. This leaves a region of medium initial data, where we have stationary points of the flow, which are the solutions to the Hull-Strominger system found in \cite{fei2017}. Therefore it is desirable to understand the behavior of this flow with medium initial data.

A subtle issue in this case is that there is an obstruction to the existence of stationary points, which comes from the ``hemisphere condition'' in our previous work \cite{fei2017}. Moreover, this hemisphere condition controls the Morse index of the Jacobi operator $-\Delta_{\hat{\omega}}+2\kappa$, which in turn gives us information about the number of unstable directions of the linearized operator of our flow at stationary points.

In principle, the obstruction of the hemisphere condition should be detected by a purely analytical understanding of our flow. Moreover, there is a kind of ``surgery'' given by composition with automorphisms of $\ccc\pp^1$ (a special form of reparametrization), which enables us to continue the flow if we encounter a singularity due to the hemisphere obstruction. This phenomenon reminds us of Hamilton's Ricci flow \cite{hamilton1982,hamilton1988,chow1991,struwe2002,phong2014} and it is an interesting problem to characterize the above geometric picture by the analytical theory of PDE. Meanwhile the freedom of reparametrization is also related to the moduli space of solutions, which is of great importance from both mathematical and string-theoretical points of view. We hope to address these questions in our future work.

\bibliographystyle{plain}

\bigskip
Department of Mathematics, Columbia University, New York, NY 10027, USA

\smallskip
tfei@math.columbia.edu, zjhuang@math.columbia.edu, picard@math.columbia.edu

\end{document}